\documentclass{amsart}
\usepackage{tabularx}
\usepackage{enumitem,amsmath,amsfonts,amssymb,xcolor}
\usepackage{latexsym,amsfonts,euscript,amssymb,gensymb}
\usepackage{amsmath}
\newtheorem{lemma}{\bf Lemma}[section]

\newtheorem{prop}[lemma]{\bf Proposition}
\newtheorem{thm}[lemma]{\bf Theorem}

\newcommand{\GL}{{\operatorname{GL}}}
\newcommand{\SL}{{\operatorname{SL}}}

\newcommand{\PSL}{{\operatorname{PSL}}}
\newcommand{\GU}{{\operatorname{GU}}}
\newcommand{\SU}{{\operatorname{SU}}}

\newcommand{\PSU}{{\operatorname{PSU}}}
\newcommand{\Sp}{{\operatorname{Sp}}}
\newcommand{\PSp}{{\operatorname{PSp}}}
\newcommand{\OO}{{\operatorname{P\Omega}}}
\newcommand{\GO}{{\operatorname{GO}}}

\DeclareMathOperator{\Aut}{Aut}

\input xy
\xyoption{all}

\title[]{On small Sylow numbers of finite groups}

\author{Xiaofang Gao}
\address{Xiaofang Gao. Departamento de Matem\'atica, Universidade de Bras\'ilia, Campus
Universit\'ario \\ Darcy Ribeiro, Bras\'ilia-DF, 70910-900, Brazil. }
\email{gaoxiaofang2020@hotmail.com}

\author{Igor Lima}
\address{Igor Lima. Departamento de Matem\'atica, Universidade de Bras\'ilia, Campus
Universit\'ario \\ Darcy Ribeiro, Bras\'ilia-DF, 70910-900, Brazil. }
\email{igor.matematico@gmail.com}

\author{Rulin Shen}
\address{Rulin Shen. Department of Mathematics, Hubei Minzu University \\ Enshi, Hubei Province,
445000, P. R. China. }
\email{shenrulin@hotmail.com}

\thanks{Project supported by the NSF of China (Grant No. 12161035)}

\date{}


\begin{document}

\setlength{\parskip}{2mm}

\maketitle

\begin{abstract}
Let $G$ be a finite group and $n_p(G)$ the number of Sylow $p$-subgroups of $G$. In this paper, we prove if $n_p(G)<p^2$ then
almost all numbers $n_p(G)$ are a power of a prime.
\end{abstract}

\section{Introduction}

Let $G$ be a finite group and $p$ a prime. We denote by $n_p(G)$ the number of Sylow $p$-subgroups of $G$, which is called Sylow $p$-number
 of $G$ (hereinafter referred to as Sylow number). The influence of the number of Sylow subgroups in finite groups on group structure is a very meaningful research topic. 
We called $n_p(G)$ a solvable Sylow number if its $\mathcal{\ell}$-part was congruent to $1$ modulo $p$ for any prime $\mathcal{\ell}$.
In 1967, Hall \cite{MHJ} studied the number of Sylow subgroups in
finite groups, and proved that solvable groups have solvable Sylow
numbers, and $22$ is never a Sylow $3$-number and $21$  a Sylow
$5$-number. In 1995, Zhang \cite{zhang} proved that a finite group $G$
is $p$-nilpotent if and only if $p$ is prime to every Sylow number
of $G$. It is easy to see
that $n_p(H)\leqslant n_p(G)$ for $H\leqslant G$, however $n_p(H)$ does not divide $n_p(G)$ in general. In 2003, Navarro \cite{GN} proved that if $G$ is
$p$-solvable, then $n_p(H)$ divides $n_p(G)$ for every $H\leqslant G$. In
2016 \cite{Litianze}, Li and Liu classified finite non-abelian simple
group with only solvable Sylow numbers. We say that a group $G$
satisfies $DivSyl(p)$ if $n_p(H)$ divides $n_p(G)$ for every $H\leqslant
G$. In 2018, Guo and Vdovin \cite{GV} generalized the results
of Navarro, and proved that $G$ satisfies $DivSyl(p)$ provided
every non-abelian composition factor of $G$ satisfies $DivSyl(p)$.
Recently, Wu \cite{Wu} proved that almost all finite simple groups do not
satisfy $DivSyl(p)$. \medskip

In this paper, we consider the numbers of Sylow subgroups $n_p(G)<p^2$, and prove that almost all numbers $n_p(G)$ are a power of a prime. Our main theorem is the following.

\begin{thm} \label{main}
Let $G$ be a finite group and $p$ a prime. Then the number $n_p(G)$ of Sylow $p$-subgroups of $G$ is less than $p^2$ if and only if $n_p(G)$ is a power of a prime being less than $p^2$ or one of 
the following cases:
\begin{enumerate}
\item $n_{13}(G)=12^2,$
\item $n_p(G)=1+p$, where $p$ is odd and not Mersenne,
\item $n_p(G)=1+(p-3)p/2$, where $p>3$ is Fermat,
\item $n_p(G)=1+(p+3)p/2$, where $p>3$ is Mersenne.
\end{enumerate}
\end{thm}

\begin{thm} \label{maincorollary}
Let $G$ be a finite group and $p$ not a Mersenne prime that is greater than $3$. Assume that the number $n_p(G)$ of Sylow $p$-subgroups of
$G$ is less than $p^2$. Then $G$ is $p$-solvable if and only if $n_p(G)$ is a power of a prime.
\end{thm}

\section{Preliminary results}

In this section we introduce and prove some useful results.

\begin{lemma}[Zsigmondy's Theorem \cite{KZ}]\label{Zsig}
Let $a$ and $n$ be integers greater $1$. Then there exists a prime divisor $q$ of $a^n-1$ such that $q$ does not divide $a^j-1$ for all $0<j<n$, except exactly the following cases:
\begin{enumerate}
    \item $n=2$, $a=2^s-1$, where $s\geqslant 2$.
    \item $n=6$, $a=2$.
\end{enumerate}
\end{lemma}

Note that the prime divisor $q$ of $a^n-1$ as described above called a primitive prime or Zsigmondy prime. 

\begin{lemma}\label{number}
 Let $p=2^n+1$ be a Fermat prime  and $q$ prime. If $p^2-3p+2=2q^a$ where $a$ is a positive integer, then $(n,p,q,a)=(1, 3, 1,1)$.
\end{lemma}

\begin{proof}
Since $p=2^n+1$, $p^2-3p+2=2^{2n}-2^n=2q^a$, this implies that $2^{n-1}(2^n-1)=q^a$. Note that $(2^{n-1}, 2^n-1)=1$ and $q^a$ has only one prime divisor, so either $2^{n-1}=1$ or $2^n-1=1$. Thus $n=q=a=1$ and $p=3$.  
\end{proof}

\begin{lemma}\label{primitive}
  Let $p$ be an odd prime and $e=min\{k\geqslant 1|q^{2k}\equiv 1(\bmod ~p)\}$. Then $p$ is a primitive prime divisor of $q^e-1$ or $q^{2e}-1$.  
\end{lemma}
\begin{proof}
   Since $p|(q^{2e}-1)$, $p$ divides $q^e+1$ or $q^e-1$. If $p|(q^e-1)$, then $e$ is odd by the minimality of $e$, this implies that $p$ is a primitive prime divisor of $q^e-1$.
   Now assume that $p$ divides $q^e+1$, then $q^e-1$ is not divisible by $p$ since $(q^e-1,q^e+1)=2$. We claim that $p$ is a primitive prime divisor of $q^{2e}-1$. If this was not the case, then there would exist a minimal integer $k$ with $k<2e$ such that $p|(q^k-1)$, this implies that $k|2e$. By the minimality of $e$, $k$ would be odd, implying $k|e$. Write $e=ks$ where $s>1$ (since $p$ does not divide $q^e-1$). Then
   $$q^e-1=q^{ks}-1=(q^k-1)((q^k)^{s-1}+(q^k)^{s-2}+\cdots +q^k+1).$$
   Note that $p|(q^k-1)$, so $p$ would divide $q^e-1$, a contradiction.
\end{proof}

\begin{lemma}[Theorem 2.1 of \cite{MHJ}]\label{normal}
  Let $G$ have a normal subgroup $K$ and let $P$ be a Sylow $p$-subgroup of $G$. Then $$n_p(G)=n_p(G/K)n_p(K)n_p(C)$$ where $C=N_{PK}(P\cap K)/P\cap K$.  
\end{lemma}

\begin{lemma}[Lemma 2.7 of \cite{Litianze}]\label{Litianze}
  Let $G$ be a finite group, $H$ a normal subgroup of $G$, and $p$ a prime. If $p$ does not divide $|G/H|$ then $n_p(H/Z(H))=n_p(G)$. 
\end{lemma}

\begin{lemma} \label{sylowstructure}
Let $S$ be a non-abelian simple group, let $P$ be a Sylow $p$-subgroup of $S$. If $n_p(S)<p^2$, then $P\cong C_p$.
\end{lemma}

\begin{proof}
Write $n:=n_p(S)$. As $n\equiv1 (\bmod~p)$ we can say $n=1+rp$ with $r<p$. Note that $n=|S:N_S(P)|$, we consider the permutation
representation of $S$ on cosets of $N_S(P)$.
Since $S$ is simple, $S$ is isomorphic to a subgroup of $S_n$. From $n<p^2$ we deduce that $S_n$ has no element of order $p^2$, and hence
the exponent of Sylow $p$-subgroup of $S_n$ is $p$.
Observe that the $p$-part of $|S_n|$ is
\begin{align*}
|S_n|_p=p^{[\frac{n}{p}]+[\frac{n}{p^2}]+[\frac{n}{p^3}]+\cdots}=p^{[\frac{1+rp}{p}]}=p^r,    \end{align*}
so we can get that the Sylow $p$-subgroup of $S_n$ is an elementary abelian $p$-group, and then $P\cong C^a_p$ for some integer $a$. Let
$P_1$ be a Sylow $p$-subgroup of $S$ distinct from $P$, then $N_{P}(P_1)=P\cap P_1$ because otherwise $N_P(P_1)\not\leqslant P_1$, $P_1$
would be a proper subgroup of  $N_P(P_1)P_1$, which would be a $p$-group since $P_1\unlhd N_P(P_1)$, contradicting $P_1$ is a Sylow
$p$-subgroup of $S$.
So the number of conjugates of $P_1$ under $P$ is $|P:P_1\cap P|<n<p^2$, thus $|P:P_1\cap P|=p$. By using \cite [Brodkey Theorem
5.28]{Isaacs} we have $P_1\cap P=1$ because $S$ is simple. Thus $P\cong C_p$.
\end{proof}

\begin{lemma}\label{cyclicmaximal}
Let $S$ be a simple group, and let $M\cong C_n.N$ be an extension of $C_n$ by $N$, which is a maximal subgroup of $S$. If the Sylow $p$-subgroup $P$ of $S$ is contained in $C_n$, then $n_p(S)=|S:M|$.
\end{lemma}

\begin{proof}
Since $P\leqslant C_n$, $P$ is a characteristic subgroup of $C_n$. Note that $C_n\unlhd M$, so $P\unlhd M$. Implying that $N_M(P)=M$. On the other hand, $P\leqslant N_M(P)\leqslant N_S(P)<S$. By the maximality of $M$, we have $N_S(P)=M$. Thus $n_p(S)=|S:M|$.
\end{proof}

\section{The number of Sylow subgroups of simple groups}
In this section, we will discuss the number of Sylow $p$-subgroups $P$ of finite simple groups and use the classification theorem for finite
simple groups to prove Theorem \ref{sylowsimple}. As usual, $\pi(S)$ denotes the set of prime divisors of $|S|$, by
$\Phi_n(q)$ we denote the $n$th cyclotomic polynomial, by function $\phi$ denotes Euler's totient function.

\begin{thm} \label{sylowsimple}
Let $S$ be a finite simple group and $n_p(S)=1+rp$ for some integer $r$. Then $n_p(S)<p^2$ if and only if one of the following holds:
\begin{enumerate}
\item $r=1$, $S\cong \PSL_2(p)$.
\item $r=11$, $S\cong \PSL_3(3)$ and $p=13$.
\item $r=\frac{p-3}{2}$, $S\cong \PSL_2(p-1)$, where $p>3$ is a Fermat prime.
\item $r=\frac{p+3}{2}$, $S\cong \PSL_2(p+1)$ with $p>3$ is a Mersenne prime.
\end{enumerate}
\end{thm}

 We will analyse the different possibilities for $S$ by using the classification of finite simple groups. Note that $n_p(S)<p^2$ implies that a Sylow $p$-subgroup of $S$ is isomorphic to $C_p$ by Lemma \ref{sylowstructure}, thus the $p$-part of $|S|$ is $p$.
 We assume $p\in \pi(S)$ such that $n_p(S)<p^2$ and let $P$ be a Sylow $p$-subgroup of $S$. We will divide the proof of Theorem \ref{sylowsimple} into several propositions.

\begin{prop}
    If $S$ is a sporadic simple group or $S\cong {}^2F_4(2)'$, then there does not exist $p\in \pi(S)$ such that $n_p(S)<p^2$.
\end{prop}

\begin{proof}
Since $S$ is a simple group, $N_S(P)<S$, it follows that there is a maximal subgroup $M$ of $S$ such that $N_S(P)\leqslant M<S$. Thus $|S:N_S(P)|\geqslant |S:M|$, it is not less than the degree of minimal permutation representation of simple group $S$. With the help of \cite{CCNPW}, $n_p(S)\geqslant p^2$, a contradiction.
\end{proof}

\begin{prop}\label{Alternatingsimple}
    Let $S\cong A_n$ with $n\geqslant 5$. Then $n_p(S)<p^2$ if and only if  $S\cong A_5$ and $p=5$.
\end{prop}

\begin{proof}
Suppose now that $S\cong A_n$ with $n\geqslant 5$. Obviously, $p\leqslant n<2p$ since $P\cong C_p$, thus $p\neq 2$. Applying \cite{LW}
we have
\begin{align*}
n_p(S)&=\frac{n!}{(n-p)!p(p-1)}=\frac{n(n-1)\cdots (n-p+2)(n-p+1)}{p(p-1)}\\&
\geqslant (n-2)(n-3)\cdots (n-(p-2))(n-(p-1))\geqslant (p-2)!.
\end{align*}
Assume $p\geqslant 7$, then $(p-2)!>p^2$, this implies that $n_p(S)>p^2$, a contradiction. Thus $p=5$. If $n=5$ then
$n_5(A_5)=6<25$ and if $n\geqslant 6$ then
$$n_p(S)=\frac{n!}{20(n-5)!}=\frac{n(n-1)(n-2)(n-3)(n-4)}{20}\geqslant \frac{6!}{20}.$$
In the second case, $n_p(S)>25=p^2$.
It remains to deal with the case $p=3$, that is $n=5$ because $3\leqslant n<6$, it is easy to check that $n_3(A_5)=10>3^2$, a
contradiction.
\end{proof}

\begin{prop}\label{linearsimple}
Let $S\cong \PSL_n(q)$ with $q=r^f$. Then $n_p(S)<p^2$ if and only if the following hold:
\begin{enumerate}
    \item $S\cong \PSL_2(p)$ and $n_p(S)=1+p$.
    \item $S\cong \PSL_3(3)$, $p=13$ and $n_{13}(S)=144$.
    \item $S\cong \PSL_2(p-1)$ and $n_p(S)=1+\frac{(p-3)p}{2}$ where $p>3$ is a Fermat prime.
    \item $S\cong \PSL_2(p+1)$ and $n_p(S)=1+\frac{(p+3)p}{2}$ where $p>3$ is a Mersenne prime.
\end{enumerate}
\end{prop}

\begin{proof}
Suppose that $S\cong \PSL_n(q)$ with $q=r^f$, we have
$$|S|=\frac{q^{\frac{n(n-1)}{2}}(q-1)(q^2-1)(q^3-1)\cdots (q^n-1)}{(q-1)(n,q-1)}.$$
Assume $(p,2r)\neq 1$. If $p=2$ then $r=2$ because if this was not the case, then $S\cong \PSL_2(q)$ with $q$ odd and $P$ would not be isomorphic to $C_2$. Therefore, $S\cong \PSL_n(2^f)$. Observe that $P\cong C_2$, so $S\cong \PSL_2(2)$, a contradiction.
If $p=r$ is odd then $S\cong \PSL_2(p)$, it is not difficult to get that $|N_S(P)|=p(p-1)/2$, thus $n_p(S)=1+p$.

Next, we assume that $(p,2r)=1$. 
Let $p$ be a primitive prime divisor of $q^e-1$, that is,
$$e=min\{k\geqslant 1|q^k\equiv 1(\bmod ~p)\}.$$ 
If $e\geqslant 2$, then $p$ does not divide $q-1$ and $n/2<e\leqslant n$ because $P\cong C_p$. By Lemma \ref{Litianze}, $n_p(\PSL_n(q))=n_p(\GL_n(q))$ since $|\GL_n(q):\SL_n(q)|$ is not divisible by $p$.
Note that $\GL_e(q)\leqslant \GL_n(q)$, thus $n_p(\GL_e(q))\leqslant n_p(\GL_n(q))<p^2$. Now, we are only concerned with the number of Sylow $p$- subgroups of $\GL_e(q)$. Let $G:=\GL_e(q)$ and $Q\in Syl_p(G)$. In
the light of \cite[Lemma 2.5]{Gross},  $|N_G(Q)|=e(q^e-1)$, implying that
\begin{align}\label{line}
n_p(G)
=\frac{q^{\frac{e(e-1)}{2}}(q-1)(q^2-1)\cdots(q^{e-1}-1)}{e}.
\end{align}
Since $p$ is a primitive prime divisor of $q^e-1$, $p|\Phi_e(q)$, it follows that
\begin{align*}
p^2\leqslant (\Phi_e(q))^2<(4q^{\phi(e)})^2
\end{align*}
by \cite[Lemma 2.1(f)] {glasby}.

If $e\geqslant 6$ then $q^{e-1}-1\geqslant e$ and $q^{\frac{e(e-1)}{2}}>16q^{2(e-1)}$. Therefore
\begin{align*}
n_p(G)\geqslant q^{\frac{e(e-1)}{2}}>16q^{2(e-1)}\geqslant (4q^{\phi(e)})^2>p^2,
\end{align*}
which is a contradiction. Thus $e\in \{5,4,3,2\}$.

If $e=5$, then by \cite[Lemma 2.1(f)] {glasby} and equality (\ref{line}), we get
\begin{align*}
n_p(G)&=\frac{q^{10}(q-1)(q^2-1)(q^3-1)(q^4-1)}{5}\geqslant q^{10}(q-1)^4(q+1)^2(q^2+1) \\&
\geqslant q^{14}>16q^8=(4q^{\Phi(5)})^2> (\Phi_5(q))^2\geqslant p^2,
\end{align*}
a contradiction. 

Assume $e=4$. If $q\geqslant 3$, then by \cite[Lemma 2.1(f)] {glasby} and equality (\ref{line}),   $$n_p(G)=\frac{q^{6}(q-1)(q^2-1)(q^3-1)}{4}
\geqslant q^8\geqslant (4q^{\phi(4)})^2>p^2$$
and if $q=2$ then $p=5$, $n_p(G)=336>5^2$. These cases are impossible.

We next assume that $e=3$, the equality (\ref{line}) implies that
$$n_p(G)=\frac{q^3(q-1)(q^2-1)}{3}\geqslant \frac{q^4(q-1)^2}{3}.$$
Note that $p^2\leqslant (\Phi_3(q))^2=(q^2+q+1)^2\leqslant 9q^4$. Obviously, if $q\geqslant 7$, then $n_p(G)\geqslant
q^4(q-1)^2/3\geqslant 9q^4\geqslant p^2$.
Since $p$ is a primitive prime divisor of $q^3-1$, we can get that if $q=5$, $4$, $3$ or $2$ then $p=31$, $7$, $13$ or $7$ and $n_p(G)=4000$, $960$,
$144$ or $8$, respectively. Therefore $(q,p)=(3,13)$ or $(q,p)=(2,7)$. Recall that $n/2<3\leqslant n$, thus $n=3$, $4$ or $5$. Therefore we only need to check the number of Sylow $13$-subgroups of  $\PSL_3(3)$, $\PSL_4(3)$ and $\PSL_5(3)$ and the number of Sylow $7$-subgroups of $\PSL_3(2)$, $\PSL_4(2)$ or $\PSL_5(2)$. By using \cite{GAP} $n_p(S)>p^2$ except $n_7(\PSL_3(2))=8$ and $n_{13}(\PSL_3(3))=144=1+11\cdot 13$.

Suppose $e=2$, then $n_p(G)=\frac{q(q-1)}{2}$ and $n=2$ or $3$. Since $p$ is a
primitive prime divisor of $q^2-1$, thus $p\mid q+1$. We write
$kp=q+1$ for some integer $k$. Assume $k\geqslant 2$. If $q\geqslant 5$ then $$n_p(G)=\frac{q(q-1)}{2}\geqslant\frac{(q+1)^2}{4}\geqslant p^2.$$ We are left to examine $q=2$, $3$ or $4$. In these  cases,   $(q,p)=(2,3)$ or $(4,5)$ because $p$ divides $q+1$ and $(p, 2r)=1$.  Therefore we only need to calculate $n_3(\PSL_3(2))$, $n_5(\PSL_2(4))$ and $n_5(\PSL_3(4)).$ 
Simple calculations by using GAP \cite{GAP} show that $S\cong \PSL_2(4)$ with $n_5(S)=6$ because $n_5(\PSL_3(4))=2016$, $n_3(\PSL_3(2))=28$, and both of them are greater than $p^2$. Now assume that $k=1$, we have $p=q+1$, thus $n_p(G)=\frac{(p-1)(p-2)}{2}<p^2$, so we cannot determine the relationship between $p^2$ and the number $n_p(S)$ of Sylow $p$-subgroups of $S$ by considering $n_p(G)$. Therefore, we need to consider the simple group $S$. Note that $p$ is an odd prime integer, so $q=2^f$ and
$p=2^f+1$ is Fermat. Moreover, $S\cong \PSL_2(2^f)$ or $\PSL_3(2^f)$. If $S\cong \PSL_2(2^f)$, then $S$ has a maximal subgroup $M$ of type $D_{2(q+1)}$. Thus there is a Sylow $p$-subgroup $P$ of $S$ such that $P\leqslant M$. Obviously, $P\in Syl_p(M)$ and $P\unlhd M$. By the maximality of $M$, $M=N_M(P)=N_S(P)$. Therefore, 
$$n_p(S)=\frac{|S|}{|M|}=\frac{q(q-1)}{2}=\frac{(p-1)(p-2)}{2}=1+rp,$$ where $r=\frac{p-3}{2}$. Let $S\cong \PSL_3(2^f)$. In the light of \cite [Lemma 1.2(1)]{Vasiliev}, the order of maximal torus of $S$ is $\frac{(q-1)^2}{(3,q-1)}$, $\frac{q^2-1}{(3,q-1)}$ or $\frac{q^2+q+1}{(3,q-1)}$, it implies that the Sylow $p$-subgroup $P$ of $S$ is contained in the maximal torus of order $\frac{q^2-1}{(3,q-1)}$. Thus, $\frac{q^2-1}{(3,q-1)}$ divides $|C_S(P)|$. By \cite [Table 4]{Vasiliev}, $S$ does not have element of order $2p$, so that the order $|C_S(P)|$ of the centralizer of $P$ in $S$ is $\frac{q^2-1}{(3,q-1)}$. 
Note that $N_S(P)/C_S(P)$ is isomorphic to a subgroup of Weyl group $S_3$. It is also isomorphic to a subgroup of $\Aut(P)$, which is a cyclic group of order $2^f$. Thus $N_P(S)/C_P(S)\cong C_2$ (Since $S$ is simple, it is not a $p$-nilpotent group, so that $|N_S(P)|\neq |C_S(P)|$). Thus $|N_S(P)|=\frac{2(q^2-1)}{(3,q-1)}$.
Therefore, $$n_p(S)=\frac{|S|}{|N_S(P)|}=\frac{q^3(q^3-1)}{2}>q^2+2q+1=p^2,$$ a contradiction.

 Now we consider $e=1$, i.e., $p\mid q-1$. Since $P\cong C_p$, $n\leqslant 3$. Let $S\cong \PSL_2(q)$. If $q\in \{5, 7, 9, 11\}$, then, by $(p,2r)=1$ and $p$ divides $q-1$, we only need to consider $n_3(\PSL_2(7))$ and $n_5(\PSL_2(11))$. By applying \cite{GAP}, we have $n_3(\PSL_2(7))$ and $n_5(\PSL_2(11))$ are greater than $p^2$. If $q\geqslant 13$ is odd, then $S$ has a maximal subgroup $M$ isomorphic to $D_{q-1}$. Clearly, $S$ has a Sylow $p$-subgroup $P$ that is contained in $C_{\frac{q-1}{2}}$. Moreover, $C_{\frac{q-1}{2}}\unlhd M$. By Lemma \ref{cyclicmaximal}, $n_p(S)=|S|/|M|=\frac{q(q+1)}{2}$. Recalled that $p|q-1$, we can say $q-1=kp$ for some integer $k$. If $k\geqslant 2$ then $n_p(S)\geqslant \frac{(q-1)^2}{k^2}=p^2$. Hence, $p=q-1$. Note that $p$ is prime and $q$ is odd, so $p=2$, contradicting the fact that $(p,2r)=1$. Suppose $q=2^f\geqslant 4$, then $M:=D_{2(q-1)}$ is a maximal subgroup of $S$. Similarly, $N_S(P)=M$ and $n_p(S)=\frac{q(q+1)}{2}$. Since $n_p(S)<p^2$ and $p|(q-1)$, $p=q-1=2^f-1$ is a Mersenne prime number. Therefore, 
 $$n_p(S)=\frac{(p+1)(p+2)}{2}=1+rp,$$ where $r=\frac{p+3}{2}$. Let $S\cong \PSL_3(q)$. Observe that $\PSL_2(q)\leqslant \PSL_3(q)$, thus we
only need to consider $\PSL_3(2^f)$ where $2^f=p+1\geqslant 4$ because otherwise $n_p(\PSL_3(q))\geqslant n_p(\PSL_2(q))\geqslant p^2$ by above discussion.
We claim that $\PSL_3(2^f)=\SL_3(2^f)$. If this was not the case, then $Z(\SL_3(2^f))$ would be isomorphic to $C_{(3,2^f-1)}=C_3$, thus
$p=3$ and $S\cong \PSL_3(4)$, by using GAP \cite{GAP} again, $n_3(\PSL_3(4))=280>9$, a contradiction. By \cite [Table 8.3]{Bray} $S$ has a maximal subgroup $M$ of type $(q-1)^2:S_3=p^2:S_3$. Contradicting the fact that $P\cong C_p$. 
\end{proof}

\begin{prop}\label{unitarysimple}
    Let $S\cong \PSU_n(q)$ where $q=r^f\geqslant 2$ and $n\geqslant 3$. Then there does not exist $p\in \pi(S)$ such that $n_p(S)<p^2$.
\end{prop}

\begin{proof}
 Suppose now that $S\cong \PSU_n(q)$. Then
\begin{align*}
|S|=\frac{q^{\frac{n(n-1)}{2}}(q+1)(q^2-1)(q^3+1)\cdots(q^n-(-1)^n)}{(n,q+1)(q+1)}.
\end{align*}
Obviously, $p$ is coprime to $2r$ since otherwise $n=2$, contradicting $n\geqslant 3$. 


If $n\geqslant 4$, then by Lemma \ref{sylowstructure}, $p$ does not divide $q+1$.
Let
$$e=min\{ k\geqslant 1\big|(-q)^k\equiv 1(\bmod ~p)\}.$$
Since $P\cong C_p$, $n/2<e\leqslant n$. Lemma \ref{Litianze} implies that $n_p(\PSU_n(q))=n_p(\GU_n(q))<p^2$ because $\SU_n(q)\unlhd \GU_n(q)$ and $|\GU_n(q):\SU_n(q)|$ is not divisible by $p$. 
In the light of \cite [Theorem 3.9] {Wilson}, $\GU_e(q)\leqslant \GU_n(q)$, so $n_p(\GU_e(q))\leqslant n_p(\GU_n(q))<p^2$. In order to prove our result, it is enough to prove that $n_p(\GU_e(q))\geqslant p^2$. For the convenience, let $G:=\GU_e(q)$ and let $Q$ denote the Sylow $p$-subgroup of $G$.
Using \cite [Theorem 3] {Vasi} we can obtain
\begin{align*}
|N_G(Q)|=|C_{(q^e-(-1)^e)_p}|\cdot |C_{(q^e-(-1)^e)_{p'}}|\cdot |C_e|=e\cdot (q^e-(-1)^e).
\end{align*}
Therefore,
\begin{align}\label{4}
 n_p(G)
=\frac{q^{\frac{e(e-1)}{2}}(q+1)(q^2-1)\cdots (q^{e-1}-(-1)^{e-1})}{e}.
\end{align}
Clearly, $q^{e-1}+1\geqslant q^{e-1}-1\geqslant e$ since $e\geqslant 3$. By the above formula (\ref{4}), we have
\begin{align}\label{su}
n_p(G)\geqslant q^{\frac{e(e-1)}{2}}(q+1)(q^2-1)\cdots (q^{e-2}-(-1)^{e-2}).
\end{align}
If $e$ is even then $p$ divides $q^e-1$ and if $e$ is odd then $p$ divides $q^e+1$. Thus, the order of $q$ in the multiplicative group $(\mathbb{Z}/p \mathbb{Z})^{\ast}$ is either $e$ or $e/2$ when $e$ is even, and it is at most $2e$ when $e$ is odd by Lemma \ref{primitive}.
Implying that the order of $q$ in this group is at most $2e$. By the properties of Euler's totient function and \cite[Lemma 2.1(f)] {glasby}, we have
$$p^2\leqslant (\Phi_{2e}(q))^2<(4q^{\phi(2e)})^2\leqslant 16q^{4e-2}.$$
We will examine $n_p(G)$ by discussing $e$. Assume $e\geqslant 9$, we have $$n_p(G)\geqslant q^{\frac{e(e-1)}{2}}\geqslant 16q^{4(e-1)}\geqslant p^2,$$ a contradiction.
If $e=8$ or $7$, then by the equality (\ref{su}), it is easy to check that $n_p(G)\geqslant 16q^{4e-2}$, these cases are impossible.
Let $e=6$, so that $p$ is a primitive prime divisor of $q^6-1$ or $q^3-1$, this follows that $p^2\leqslant (4q^{\phi(3)})^2=(4q^{\phi(6)})^2=16q^4$ with the help of \cite[Lemma 2.1(f)] {glasby}. On the other hand,
 $$n_p(G)\geqslant q^{15}(q+1)(q^2-1)(q^3+1)(q^4-1)>q^{15}>16q^4>p^2,$$
 by the equality (\ref{4}), contradicting $n_p(G)<p^2$. 
Assume $e=5$, the formula (\ref{4}) implies that 
 \begin{align*}
  n_p(G)&=\frac{q^{10}(q+1)(q^2-1)(q^3+1)(q^4-1)}{5}
 \geqslant q^{16}(q-1)^2(q^2-q+1).
 \end{align*}
Recall that $p^2\leqslant (\Phi_{10}(q))^2< (4q^{\phi(10)})^2=16q^8$ by \cite[Lemma 2.1(f)] {glasby}. Thus, $n_p(G)\geqslant
q^{16}(q-1)^2(q^2-q+1)>16q^8>p^2$. Assume now that $e=4$, in other words, $p$ is a primitive prime divisor of $q^4-1$ or $q^2-1$, so that $p^2\leqslant (4q^\phi(2))^2\leqslant (4q^{\phi(4)})^2=16q^4$ by using \cite[Lemma 2.1(f)] {glasby} again. In the light of the
formula (\ref{4}),
 \begin{align*}
n_p(G)=\frac{q^6(q+1)(q^2-1)(q^3+1)}{4}
\geqslant q^6(q^3+1)(q-1)>16q^4>p^2, 
 \end{align*}
which yields a contradiction.
Finally, assume $e=3$, then $n_p(G)=\frac{q^3(q+1)(q^2-1)}{3}$ by the formula (\ref{4}). Since $p^2\leqslant (\Phi_6(q))^2=(q^2-q+1)^2\leqslant 3q^4$, we have
\begin{align*}
n_p(G)=\frac{q^3(q+1)(q^2-1)}{3}\geqslant 3q^4\geqslant p^2
\end{align*}
if $q\geqslant 4$. Finally, if $q=3$ then $p=7$ by the definition of $e$ and $(p,2r)=1$. Moreover, $n_7(G)=288$. If $q=2$ then $p=3$
and $n_p(G)=24$. Both of them are greater than $p^2$, we deduce a contradiction.

Next, we suppose that $n=3$, that is, $S\cong \PSU_3(q)$ where $q>2$. In the light of \cite [Lemma 2.2]{Wu}, $n_p(S)=n_p(\SU_3(q))<p^2$. Write $H:=\SU_3(q)$ and $Q$ the Sylow $p$-subgroup of $H$, we have 
$$|H|=q^3(q-1)(q+1)^2(q^2-q+1).$$ We claim that $p\nmid q+1$. If this was not the case, then $p=3$ because the Sylow $p$-subgroup of $S$ is isomorphic to $C_p$, since $H$ has a maximal subgroup of type $(q+1)^2:S_3$ (see \cite [Table 8.5]{Bray}), this implies that $27$ would be a divisor of $|H|$, and hence $3^2\big ||S|$, contradicting $P\cong C_p$. 
Assume $p\mid q^2-q+1$, by \cite [Table 8.5]{Bray}, $H$ has a maximal subgroup $M$ of type $(q^2-q+1):3$, thus $H$ has a Sylow $p$-subgroup $Q$ such that $Q\leqslant M$ because $(|H:M|, q^2-q+1)=1$. Since $Q\unlhd_c C_{q^2-q+1}\unlhd M$, $Q$ is normal in $M$. It is obvious that $Q$ is not normal in $H$ since $p$ does not divide $q+1$ and $H$ has an unique normal subgroup $Z(H)$ with $Z(H)=(q+1,3)$. 
 Thus $M=N_H(Q)\geqslant N_M(Q)=M$ by the maximality of $M$, therefore
\begin{align*}
 n_p(H)=|H:((q^2-q+1):3)|=\frac{q^3(q-1)(q+1)^2}{3}\geqslant q^4(q-1).
\end{align*}
Clearly, if $q\geqslant 4$ then $q^4(q-1)\geqslant 3q^4\geqslant (q^2+1)^2\geqslant (q^2-q+1)^2\geqslant p^2$ and if $q=3$ then $p=7$, $n_7(H)=288>49$, a contradiction.
Finally, if $p\mid q-1$, then $p^2\leqslant (q-1)^2\leqslant 2q^2$. By applying \cite [Table 8.5] {Bray} again, $H$ has a maximal subgroup $M$ isomorphic to $[q^3]:(q^2-1)$. Note that $Q$ is a Sylow $p$-subgroup of $M$ because $|H:M|$ is not divisible by $p$. We claim that $N_M(Q)=C_{q^2-1}$. In fact, we assume that $N:=N_M(Q)\cap [q^3]$ is a nontrivial $r$-group, thus $N\unlhd N_M(Q)$ since $[q^3]\unlhd M$. This implies that $Q\leqslant N_M(Q)\leqslant N_M(N)$. Observe that $(|N|, |Q|)=1$, so $mn=nm$ for any $m\in Q$, $n\in N$, that is, $H$ has an element with order $rp$, where $p$ is a primitive prime divisor of $q-1$. Since $p$ is an odd prime and $p\mid q-1$, $q-1\neq 2^k$ for some integer $k$. By applying \cite [Table 4] {Vasiliev}, $H$ does not have an
element of order $rp$, a contradiction. Therefore,  $N=1$, and $N_M(Q)=C_{q^2-1}$. Consequently,
$$n_p(M)=|M:N_M(Q)|=q^3\geqslant 2q^2\geqslant p^2.$$ 
This implies that $n_p(H)\geqslant n_p(M)\geqslant p^2$, a
contradiction.
\end{proof}

\begin{prop}\label{symplecticsimple}
If $S\cong \PSp_{2n}(q)$ where $q=r^f\geqslant 2$ and $n\geqslant 2$, then there does not exist $p\in \pi(S)$ such that $n_p(S)<p^2$.
\end{prop}

\begin{proof}
Assume $S\cong \PSp_{2n}(q)$, where $n\geqslant 2$ and $q=r^f\geqslant 2$.
Then $$|S|=\frac{q^{n^2}(q^2-1)(q^4-1)\cdots(q^{2n}-1)}{(2,q-1)}.$$
Since $P\cong C_p$, $(p, 2r)=1$. Set 
$$e=min\{k\geqslant 1\big| q^{2k}\equiv 1(\bmod ~p)\},$$
then $p$ is a primitive prime divisor of $q^e-1$ or $q^{2e}-1$ by Lemma \ref{primitive}. Clearly, $n/2<e\leqslant n$ (because $P\cong C_p$). By \cite [Theorems 3.7 and 3.8] {Wilson}, 
$S$ has a subgroup isomorphic to
$\Sp_{2e}(q)$. Note that $n_p(\Sp_{2e}(q))\leqslant n_p(S)<p^2$, so we can prove our proposition by considering $n_p(\Sp_{2e}(q))$. For convenience, let $G:=\Sp_{2e}(q)$ and $Q\in Syl_p(G)$, then $$|G|=q^{e^2}(q^2-1)(q^4-1)\cdots (q^{2e}-1).$$

We first suppose that $p$ is a primitive prime divisor of $q^{2e}-1$, Clearly, $e=2e/(2e,2)$. By \cite [Theorem 2.6] {Gross},
\begin{align*}
 |N_G(Q)|=2e(q^e+(-1)^{2e})=2e(q^e+1),
\end{align*}
so that
\begin{align}\label{5}
n_p(G)
=\frac{q^{e^2}(q^2-1)(q^4-1)\cdots(q^{2e-2}-1)(q^e-1)}{2e}.
\end{align}
If $e\geqslant 4$, then $q^e-1\geqslant 2e$ and $q^{e^2}(q^2-1)(q^4-1)(q^6-1)\geqslant q^{e^2+3}(q^6-1)\geqslant 16q^{4e-2}$, thus
\begin{align*}
 n_p(G)\geqslant q^{e^2}(q^2-1)(q^4-1)(q^6-1)\geqslant 16q^{4e-2}.
\end{align*}
Recall that $p$ is a primitive prime divisor of $q^{2e}-1$,
\begin{align*}
p^2\leqslant (\Phi_{2e}(q))^2<(4q^{\phi(2e)})^2\leqslant 16q^{4e-2}
\end{align*}
by \cite[Lemma 2.1(f)] {glasby}.
This follows that $n_p(G)\geqslant 16q^{4e-2}\geqslant p^2$, a contradiction.
Assume $e=3$ then 
$$n_p(G)=\frac{q^{9}(q^2-1)(q^4-1)(q^3-1)}{6}\geqslant q^{12}.$$ 
\cite[Lemma 2.1(f)] {glasby} shows that $p^2\leqslant (\Phi_6(q))^2\leqslant
(4q^2)^2=16q^{4}\leqslant q^{12}$, thus $p^2\leqslant n_p(G)$.
We now assume $e=2$ then
\begin{align*}
n_p(G)=\frac{q^4(q^2-1)^2}{4}
\end{align*}
by the formula (\ref{5}) and $p^2\leqslant (\Phi_4(q))^2\leqslant (4q^{\phi(4)})^2=16q^4$. It is not difficult to get that $n_p(G)\geqslant p^2$ if $q\geqslant 3$.
We are left to examine the case $q=2$, meaning that $p$ is a primitive prime divisor of $2^4-1$, and thus $p=5$ and $n_p(G)=36$. Therefore, $n_p(G)>25=p^2$. Finally, if $e=1$, then $n=1$ since
$n<2e\leqslant 2n$, thus $S\cong \PSp_2(q)\cong \PSL_2(q)$, which has been considered in the case of linear groups.

We next suppose that $p$ is a primitive prime divisor of $q^e-1$, thus $e$ is odd by the minimality of $e$. Observe that
$e=e/(e,2)$. Using \cite [Theorem 2.6] {Gross} again, we obtain
\begin{align*}
n_p(G)&=\frac{q^{e^2}(q^2-1)(q^4-1)\cdots(q^{2e-2}-1)(q^{2e}-1)}{2e(q^e-1)}\\&=\frac{q^{e^2}(q^2-1)(q^4-1)\cdots (q^{2e-2}-1)(q^e+1)}{2e}.
\end{align*}
If $e\geqslant 3$ then $q^e+1\geqslant q^e-1\geqslant 2e$, so that $$n_p(G)\geqslant q^{e^2}(q^2-1)(q^4-1)\geqslant (q^2+1)(q+1)^2q^{e^2}\geqslant 16q^{2e-2}.$$ 
Obviously, $p^2\leqslant (\Phi_e(q))^2<
(4q^{\phi(e)})^2\leqslant 16q^{2e-2}$. Thus $n_p(G)\geqslant p^2$.
Since $e$ is odd, thus $e=1$ and then $n=1$. In this case $S\cong \PSp_2(q)\cong \PSL_2(q)$, which has been considered before.
\end{proof}

\begin{prop}\label{orthoddsimple}
    If $S\cong \OO_{2n+1}(q)$ where $q=r^f$ is odd and $n\geqslant 2$, then there does not exist $p\in \pi(S)$ such that $n_p(S)<p^2$.
\end{prop}

\begin{proof}
Suppose $S\cong \OO_{2n+1}(q)$ where $q=r^f$ is odd, then $|\OO_{2n+1}(q)|=|\PSp_{2n}(q)|$. If $p=2$, then $n=1$ since $P\cong C_2$, a contradiction. If $p\neq 2$ then, by \cite[Corollary 2.2]{Ahan},
$|N_{\OO_{2n+1}(q)}(P_1)|=|N_{\PSp_{2n}(q)}(P_2)|$ where $P_1$, $P_2$ are Sylow $p$-subgroups of $\OO_{2n+1}(q)$ and $\PSp_{2n}(q)$,
respectively. Thus $n_p(\OO_{2n+1}(q))=n_p(\PSp_{2n}(q))\geqslant p^2$ by
Proposition \ref{symplecticsimple}, a contradiction. 
\end{proof}


\begin{prop}\label{orthaddsimple}
Assume $S\cong \OO^+_{2n}(q)$, where $q=r^f$ and $n\geqslant 3$. Then there does not exist $p\in \pi(S)$ such that $n_p(S)<p^2$.
\end{prop}

\begin{proof}
Let $S\cong \OO^+_{2n}(q)$, then 
$$|S|=\frac{q^{n(n-1)}(q^n-1)(q^2-1)(q^4-1)\cdots(q^{2n-2}-1)}{(4,q^n-1)}.$$
It is not difficult to check that $(p,2r)=1$ by Lemma \ref{sylowstructure}.
Since $\Omega^+_{2n}(q)\unlhd \GO^+_{2n}(q)$ and $p$ does not divide $|\GO^+_{2n}(q):\Omega^+_{2n}(q)|$, $n_p(\GO^+_{2n}(q))=n_p(S)<p^2$ with the help of Lemma \ref {Litianze}.
Let $G:=\GO^+_{2n}(q)$. Then $$|G|=2q^{n(n-1)}(q^n-1)(q^2-1)(q^4-1)\cdots(q^{2n-2}-1).$$ Set
\begin{align*}
e=min\{k\geqslant 1\big|q^{2k} \equiv 1(\bmod ~p)\}.
\end{align*}
Since $|G:S|$ is not divisible by $p$, the Sylow $p$-subgroup $Q$ of $G$ is also a cyclic subgroup of order $p$ by Lemma \ref{sylowstructure}. Therefore, $n-1<2e\leqslant 2n-2$ and $p$ is a primitive prime divisor of $q^e-1$ or $q^{2e}-1$ by Lemma \ref{primitive}.

Assume that $p$ is a primitive prime divisor of $q^{2e}-1$. In the light of \cite [Theorem 3.12] {Wilson}, $G$ has a maximal subgroup of
type $\GO^+_{2k}(q)\times \GO^+_{2m}(q)$, with $k+m=n$ and $0<k<n$. If $e<n-1$, then $G$ has a subgroup $H$ isomorphic to $\GO^+_{2(e+1)}(q)$. Obviously, the Sylow $p$-subgroup of $H$ is not trivial because
\begin{align*}
|H|=2q^{(e+1)e}(q^{e+1}-1)(q^2-1)(q^4-1)\cdots (q^{2e}-1)
\end{align*}
and $p\mid q^{2e}-1$. Note that $p$ is a primitive prime divisor of $q^{2e}-1$ and $n_p(G)\geqslant n_p(H)=n_p(\OO^+_{2(e+1)}(q))$.
Let $S_1$ denote $\OO^+_{2(e+1)}(q)$ and $Q_1$ denote its Sylow $p$-subgroup.
In the light of \cite [Lemma 2.4]{Ahan} we have
\begin{align*}
|N_{S_1}(Q_1)|=2e(q^e+1)(q+1)/(q^{e+1}-1,4),
\end{align*}
so that
\begin{align}\label{6}
n_p(S_1)
=\frac{q^{(e+1)e}(q^{e+1}-1)(q-1)(q^4-1)(q^6-1)\cdots (q^{2e-2}-1)(q^e-1)}{2e}.
\end{align}
Recall that $p^2\leqslant (\Phi_{2e}(q))^2<(4q^{\phi(2e)})^2\leqslant 16q^{4e-2}$.
Let $e\geqslant 4$, we have $q^{(e+1)e}(q^{e+1}-1)(q^4-1)\geqslant (q^{e+1}-1)q^{e^2+e+3}\geqslant 16q^{4e-2}$ and $q^e-1\geqslant 2e$. Thus
\begin{align*}
 n_p(S_1)\geqslant q^{(e+1)e}(q^{e+1}-1)(q^4-1)\geqslant 16q^{4e-2}\geqslant p^2,
\end{align*}
therefore, $n_p(S)=n_p(G)\geqslant n_p(H)=n_p(S_1)>p^2$, we get a contradiction. 
Next if $e=3$ then  
\begin{align*}
n_p(S_1)=q^{12}(q^4-1)(q-1)(q^4-1)(q^3-1)/6
\end{align*}
and if $e=2$ then
\begin{align*}
n_p(S_1)=q^6(q^3-1)(q-1)(q^2-1)/4
\end{align*} 
by the formula (\ref{6}). By $p^2\leqslant (\Phi_{2e}(q))^2<(4q^{\phi(2e)})^2$, it is easy to check that $n_p(S_1)\geqslant p^2$ unless $e=q=2$, in which case $p=5$ and $n_p(S_1)=336>5^2$ (since $p$ is a primitive prime divisor of $q^4-1$). Hence
$n_p(G)\geqslant n_p(H)=n_p(S_1)\geqslant p^2$, a contradiction. Finally let $e=1$, then from $n-1<2e\leqslant 2n-2$ deduce that $n=2$, this contracts the
assumption. Next if $e=n-1$, that is $p$ is a primitive prime divisor of $q^{2(e-1)}-1$. We can use \cite [Lemma 2.4] {Ahan} again to calculate $n_p(S)$ directly. We have
\begin{align*}
|N_S(P)|=2(q^{n-1}+1)(q+1)(n-1)/(4,q^n-1),
\end{align*}
and hence
\begin{align}\label{7}
 n_p(S)
=\frac{q^{n(n-1)}(q^n-1)(q-1)(q^4-1)\cdots(q^{2n-4}-1)(q^{n-1}-1)}{2(n-1)}.
\end{align}
If $n\geqslant 5$, then $q^{n-1}-1\geqslant 2(n-1)$, so that
\begin{align*}
n_p(S)\geqslant q^{n(n-1)}(q^n-1)\geqslant 16q^{4n-4}\geqslant (4q^{\phi(2n-2)})^2> (\Phi_{2n-2}(q))^2\geqslant p^2,
\end{align*}
a contradiction. Similarly to the previous case, 
if $n=4$ then
\begin{align*}
n_p(S)=q^{12}(q^4-1)(q-1)(q^3-1)(q^4-1)/6
\end{align*}
and if $n=3$ then 
\begin{align*}
n_p(S)
=q^6(q^3-1)(q-1)(q^2-1)/4
\end{align*}
by the quality (\ref{7}).
Using \cite[Lemma 2.1(f)] {glasby}, we have $p^2<16q^4$, this implies that $n_p(S)>p^2$ except for $(n,q)=(3,2)$. Suppose $(n,q)=(3,2)$, then $p$ is a primitive prime divisor of $2^4-1$, implying that $p=5$ and $n_p(S)=336>p^2$, a contradiction.

Next, we suppose that $p$ is a primitive prime divisor of $q^e-1$. Then, by the minimality of $e$, $e$ is odd. Moreover, if $p$ divides $q^n-1$, then $n$ must be odd because the Sylow $p$-subgroup of $G$ is a cyclic group of order $p$ and $n<2n-2$. Thus, in this case, $p$ is a primitive prime divisor of $q^n-1$. Therefore, $e\leqslant n$. \cite[Theorem 3.12] {Wilson} implies that $G$ admits a subgroup $H$ of type $\GL_n(q)$. Note that
\begin{align*}
 |H|=q^{n(n-1)/2}(q-1)(q^2-1)(q^3-1)\cdots(q^n-1)
\end{align*}
and $p\mid q^e-1$, where $e\leqslant n$, thus the Sylow $p$-subgroup of $G$ is contained in $H$. In particular, $$n_p(G)\geqslant n_p(H)\geqslant n_p(\SL_n(q))=n_p(\PSL_n(q)).$$ By Proposition \ref{linearsimple}, we can get $n_p(G)\geqslant p^2$ unless $(n,p,q)=(3,13,3)$ and $n=2$, namely $G=\GO^+_{6}(3)$ and $S=\OO^+_6(3)\cong\PSL_4(3)$ (Since $n\geqslant 3$). In the light of Proposition \ref{linearsimple}, $n_{13}(S)\geqslant 13^2$, a contradiction.
\end{proof}

\begin{prop}\label{orthsubsimple}
Assume $S\cong \OO^-_{2n}(q)$, where $q=r^f$ and $n\geqslant 2$. Then there does not exist $p\in \pi(S)$ such that $n_p(S)<p^2$.
\end{prop}

\begin{proof}
Let $S\cong \OO^-_{2n}(q)$. Then 
$$|S|=\frac{q^{n(n-1)}(q^n+1)(q^2-1)(q^4-1)\cdots (q^{2n-2}-1)}{(4,q^n+1)}.$$
 By Lemma \ref{sylowstructure}, we can get that $p$ is coprime to $2r$, that is $p$ is an odd integer.
Similarly as Proposition \ref{orthaddsimple}, $$n_p(\OO^{-}_{2n}(q))=n_p(\Omega^-_{2n}(q))=n_p(\GO^-_{2n}(q))<p^2$$ because $|\GO^-_{2n}(q):\Omega^-_{2n}(q)|$ is not divisible by $p$. Now we are concerned with studying $n_p(\GO^-_{2n}(q))$. Let $G:=\GO^-_{2n}(q)$, then
$$|G|=2q^{n(n-1)}(q^n+1)(q^2-1)(q^4-1)\cdots (q^{2n-2}-1).$$
In particular, the Sylow $p$-subgroup $Q$ of $G$ is also isomorphic to $C_p$ because $p$ does not divide $|\GO^-_{2n}(q):S|$. Set
\begin{align*}
e=min\{k\geqslant 1\big| q^{2k}\equiv 1(\bmod ~p)\},
\end{align*}
Thus $n-1<2e\leqslant 2n-2$ by Lemma \ref{sylowstructure}, and $p$ is a primitive prime divisor of $q^e-1$ or $q^{2e}-1$ with the help of Lemma \ref{primitive}. Let $e=1$, then, $n=2$ since $n-1<2e\leqslant 2n-2$. Thus, $G\cong \GO^-_{4}(q)$ and $S\cong \OO^-_{4}(q)\cong \PSL_2(q^2)$. Moreover, $p|(q^2-1)$. By Proposition \ref{linearsimple} we have $S\cong \PSL_2(p-1)$ or $S\cong \PSL_2(p+1)$. Thus $q^2=p-1$ or $q^2=p+1$. Since $q^2-1$ is divisible by $p$ and $(p,2r)=1$, $p=q^2-1$, thus $q=2$ and $p=3$, and hence $S\cong \OO^-_4(2)\cong \PSL_2(4)$. By Proposition \ref{linearsimple} we have $n_3(S)>p^2$, a contradiction. In the rest of the proof we will assume that $e\geqslant 2$.

Assume $p$ is a primitive prime divisor of $q^e-1$, then $e\leqslant n-1$ and $e$ is odd by the minimality of $e$, so that $e\geqslant 3$. \cite[Theorem 3.11]{Wilson} shows that $G$ has a maximal subgroup of type $\GO^+_{2k}(q)\times \GO^-_{2m}(q)$, with $k+m=n$ and $0<k<n$. Thus, we can consider the subgroup $H$ of type $\GO^+_{2e}(q)$. Note that $p|(q^e-1)$ and
$$|H|=2q^{e(e-1)}(q^e-1)(q^2-1)(q^4-1)\cdots (q^{2e-2}-1),$$
so the Sylow $p$-subgroup of $G$ is contained in $H$.  Moreover, $n_p(G)\geqslant n_p(H)$. By Proposition \ref{orthaddsimple}, $n_p(G)\geqslant n_p(H)=n_p(\OO^+_{2e}(q))\geqslant p^2$ since $e\geqslant 3$, a contradiction.

Now suppose that $p$ is a primitive prime divisor of $q^{2e}-1$. Then $e\leqslant n-1$. Assume $e<n-1$, we consider the subgroup $H$ of type $\GO^+_{2(e+1)}(q)$. Note that $p|(q^{2e}-1)$ and 
$$|H|=2q^{(e+1)e}(q^{e+1}-1)(q^2-1)(q^4-1)\cdots (q^{2e}-1).$$ 
Thus $H$ contains the Sylow $p$-subgroup of $G$. In particular, $e+1\geqslant 3$. By using Proposition \ref{orthaddsimple}, $$n_p(G)\geqslant n_p(H)=n_p(\OO^+_{2(e+1)}(q))\geqslant p^2$$ (Since $p$ does not divide $|H:\Omega^+_{2(e+1)}(q)|$, we have $n_p(H)=n_p(\OO^+_{2(e+1)}(q))$), a contradiction. Suppose that $e=n-1$, that is, $p$ is a primitive prime divisor of $q^{2n-2}-1$, so that $p|(q^{n-1}+1)$.
Assume $q$ is odd. By using \cite[Theorem 3.11]{Wilson} $G$ has a maximal subgroup $M$ of type $[q^{2(n-1)}].(C_{q-1}\times \GO^-_{2(n-1)}(q))$. Since $Q\cong C_p$ and $p||M|$ (Recall that $Q\in Syl_p(G)$), any Sylow $p$-subgroup of $M$ is a Sylow $p$-subgroup of $G$. 
For convenience, we write $$H:=[q^{2(n-1)}], \quad K:=C_{q-1}\times \GO^-_{2(n-1)}(q)$$ and $T\in Syl_p(M)$. We claim that $H\cap N_M(T)=1$. If this was not the case,
then $N:=H\cap N_M(T)$ would not be a trivial $r$-group, and $N\unlhd N_M(T)$ since $H\unlhd M$. We deduce that $T\leqslant N_M(T)\leqslant N_M(N)$. Thus the commutator of $T$ and $N$ would be trivial because $(|N|,|T|)=1$ (Since $N^T=N$, $T^N=T$ and $(|N|,|T|)=1$). Therefore $M$ would have an element of order $rp$, contradicting \cite [Table 4]{Vasiliev}, which shows that $M$ does not have an element of order $rp$. Hence $H\cap N_M(T)=1$ and $N_M(T)\leqslant K$. It follows that
\begin{align*}
 n_p(M)=|M:N_M(T)|\geqslant |M:K|=q^{2(n-1)}.
\end{align*}
Recall that $p\mid q^{n-1}+1$, we write $kp=q^{n-1}+1$ for some positive integer $k$.
Since both $p$ and $q$ are odd numbers, $k\geqslant 2$. Hence
\begin{align*}
 p^2=\frac{q^{2n-2}+2q^{n-1}+1}{k^2}\leqslant \frac{3q^{2n-2}}{k^2}\leqslant q^{2(n-1)}\leqslant n_p(M).
\end{align*} 
Therefore, $n_p(G)\geqslant n_p(M)\geqslant p^2$, a contradiction. Assume $q$ is even, i.e., $q=2^f$. By \cite [Theorem 3.11]{Wilson}, $G$ has a maximal subgroup of type $\Sp_{2(n-1)}(q)$. Obviously, $Q\leqslant M$ and $$n_p(G)\geqslant n_p(\Sp_{2(n-1)}(q))=n_p(\PSp_{2(n-1)}(q))>p^2$$ unless $n=2$ by Proposition \ref{symplecticsimple}. Finally, let $n=2$, then $G\cong \GO^-_{4}(q)$ and $S\cong \OO^-_{4}(q)\cong \PSL_2(q^2)$.  By Proposition \ref{linearsimple}, $q^2=p-1$ or $q^2=p+1$, this implies that $p=q^2+1$ or $p=q^2-1$, contradicting the fact that $p|(q+1)$.

Finally, we suppose that $p\mid q^n+1$. If $n$ is odd, then $G$ has a subgroup $H:=\GU_n(q)$ by \cite[Theorem 3.11]{Wilson}. Obviously,
$H$ contains the Sylow $p$-subgroup of $G$ and $$n_p(G)\geqslant n_p(H)\geqslant n_p(\PSU_n(q))\geqslant p^2$$ by Proposition \ref{unitarysimple}, a contradiction. Assume $n$ is even. Write $n=2^a\cdot d$ with $d$ odd. \cite[Theorem 3.11]{Wilson} implies 
that $G$ has a subgroup $H$ isomorphic to $\GU_d(q^{2^a})$. It is easy to check that the Sylow $p$-subgroup of $H$ is not trivial and $n_p(G)\geqslant n_p(H)$. If $d\geqslant 3$, then by Proposition \ref{unitarysimple}, 
$$n_p(G)\geqslant n_p(H)\geqslant n_p(\PSU_d(q^{2^a}))\geqslant
p^2,$$ a contradiction. Let $d=1$, that is $G\cong \GO^-_{2\cdot 2^a}(q)$ and $p|(q^{2^a}+1)$. Similarly as above, $G$ has a subgroup $H$ of type $\GO^{-}_{2\cdot 2}(q^{2^{a-1}})$. Note that $$n_p(S)=n_p(G)\geqslant n_p(H) =n_p(\Omega^{-}_4(q^{2^{a-1}}))=n_p(\PSL_2(q^{2^a})).$$ Since $n_p(S)<p^2$, by Proposition \ref{linearsimple}, $q^{2^a}+1=p$ is a Fermat prime number, so that $q$ is even. On the other hand, $G$ also has a subgroup $M$ of type $\GO^-_8(u)$ where $u=q^{2^{a-2}}$ and it is easy to see that $P$ is contained in the simple group $A:=\OO^-_8(u)$. Moreover, $p=u^4+1$. Since $n_p(S)=n_p(G)\geqslant n_p(M)=n_p(A)$ and every maximal torus of $A$ has order
$$\frac{1}{(2,q^4+1)}(q^{n_1}-1)(q^{n_2}-1)\cdots (q^{n_k}-1)(q^{l_1}+1)(q^{l_2}+1)\cdots (q^{l_t}+1)$$ 
for appropriate partition $4=n_1+n_2+\cdots+n_k+l_1+l_2+\cdots+l_t$ of $4$, where $t$ is odd. This implies that $P$ is a maximal torus of $A$ and $C_A(P)=P$. Note that $N_A(P)/C_A(P)$ is isomorphic to a subgroup of $\Aut(P)\cong C_{u^4}$, thus $|N_A(P)|\leqslant u^4(u^4+1)$. It follows that 
\begin{align*}
n_p(A)&=\frac{|A|}{|N_A(P)|}\geqslant \frac{u^{12}(u^4+1)(u^2-1)(u^4-1)(u^6-1)}{u^4(u^4+1)}\\&
=u^8(u^2-1)(u^4-1)(u^6-1)>u^{17}.
\end{align*}
Since $p=u^4+1$, $p^2=(u^4+1)^2\leqslant 3u^8<u^{17}$, it implies that $n_p(A)>p^2$. Therefore, $n_p(S)\geqslant n_p(A)>p^2$, a contradiction.
\end{proof}

Next we will consider the exceptional group of Lie type.  

\begin{prop}\label{suzukisimple}
Let $S\cong Sz(q)$ where $q=2^{2m+1}\geqslant 8$. Then there does not exist $p\in \pi(S)$ such that $n_p(S)<p^2$.
\end{prop}

\begin{proof}
Let $S\cong Sz(q)$. Then $$|S|=q^2(q-1)(q-\sqrt{2q}+1)(q+\sqrt{2q}+1).$$ Since $P\cong C_p$, $p\neq 2$. Assume $p\mid (q-1)$. By \cite[Table 8.16]{Bray} $S$ has a maximal subgroup $M$ of type $D_{2(q-1)}$. Obviously, $P\leqslant C_{q-1}\unlhd D_{2(q-1)}$. Thus 
$$n_p(S)=\frac{|S|}{|M|}=\frac{q^2(q^2+1)}{2}$$
with the help of Lemma \ref{cyclicmaximal}. Note that $p^2\leqslant (q-1)^2$ and $q\geqslant 8$. Therefore $n_p(S)=\frac{q^2(q^2+1)}{2}\geqslant p^2$, a contradiction. Suppose $p\mid (q-\sqrt{2q}+1)$, then 
$$n_p(S)=\frac{q^2(q-1)(q^2+\sqrt{2q}+1)}{4}\geqslant q^2(q+\sqrt{2q}+1)$$ by \cite[Table 8.16]{Bray} and Lemma \ref{cyclicmaximal}. Since $p\mid (q-\sqrt{2q}+1)$, $p^2\leqslant (q-\sqrt{2q}+1)^2\leqslant (q+1)^2\leqslant 4q^2$. Implying that $n_p(S)\geqslant q^2(q+\sqrt{2q}+1)\geqslant 4q^2\geqslant p^2$ since $q\geqslant 8$, it is impossible.
Finally, we assume that $p\mid (q+\sqrt{2q}+1)$. Then $p^2\leqslant (q+\sqrt{2q}+1)^2\leqslant 5q^2$. Similarly as above, 
$$n_p(S)=\frac{q^2(q-1)(q-\sqrt{2q}+1)}{4}=\frac{q^2((q-\sqrt{2q})q+(\sqrt{2q}-1))}{4}\geqslant q^3$$
because $q-\sqrt{2q}\geqslant 4$, it follows that $n_p(S)\geqslant q^3\geqslant 5q^2\geqslant p^2$, a contradiction.
\end{proof}

\begin{prop}\label{reesimple}
Assume $S\cong G_2(q)$ where $q\geqslant 3$. Then there does not exist $p\in \pi(S)$ such that $n_p(S)<p^2$.
\end{prop}

\begin{proof}
Let $S\cong G_2(q)$. Then $$|S|=q^6(q-1)^2(q+1)^2(q^2+q+1)(q^2-q+1).$$ Since $P\cong C_p$, $p$ is coprime to $q(q^2-1)$. Thus, $p\mid (q^2+q+1)$ or $p\mid (q^2-q+1)$.  By \cite [Tables 8.30, 8.41 and 8.42] {Bray} $S$ has subgroups $\SU_3(q)$ and $\SL_3(q)$.
Assume $p$ divides $q^2+q+1$, then the Sylow $p$-subgroup $P$ of $S$ is contained in $\SL_3(q)$, so that $n_p(S)\geqslant n_p(\SL_3(q))=n_p(\PSL_3(q))\geqslant p^2$ unless $S\cong G_2(3)$ by \cite [Lemma 2.2]{Wu} and Proposition \ref{linearsimple}. Suppose $S\cong G_2(3)$, then $p=13$. Note that $S$ is simple, so $P$ is not normal in $G$, implying that there exists a maximal subgroup $M$ of $S$ such that $N_S(P)\leqslant M$, thus $n_p(S)=|S:N_S(P)|\geqslant |S:M|$. By applying \cite{CCNPW}, $n_p(S)\geqslant |S:M|>13^2$, it is impossible. Suppose $p$ divides $q^2-q+1$. Then $P\leqslant \SU_3(q)$ and $n_p(S)\geqslant n_p(\SU_3(q))=n_p(\PSU_3(q))\geqslant p^2$ with the help of \cite [Lemma 2.2]{Wu} and Proposition \ref{unitarysimple}, a contradiction.    
\end{proof}

\begin{prop}\label{R(q)simple}
If $S\cong {}^2G_2(q)$ where $q=3^{2m+1}\geqslant 27$, then there does not exist $p\in \pi(S)$ such that $n_p(S)<p^2$.
\end{prop}

\begin{proof}
Suppose $S\cong {}^2G_2(q)$. Then $$|S|=q^3(q+1)(q-1)(q-\sqrt{3q}+1)(q+\sqrt{3q}+1).$$ We first suppose that $p$ divides $q+1$. By \cite[Table 8.43]{Bray}, $S$ has a maximal subgroup $H$ of type $(2^2 \times D_{\frac{q+1}{2}}):3$. Since $|S:H|$ is not divisible by $p$, the Sylow $p$-subgroup $P$ of $S$ is contained in $H$. Moreover, $P\leqslant D_{\frac{q+1}{2}}$. Clearly, $P$ is normal in $D_{\frac{q+1}{2}}$, so $P$ is also normal in $K:= 2^2 \times D_{\frac{q+1}{2}}$. This implies that $K = N_K(P) \leqslant N_S(P) < S$. Thus $N_S(P)=K$ or $H$, and 
$$n_p(S)=|S:N_S(P)|\geqslant |S:H|=\frac{q^3(q^2-q+1)(q-1)}{6}.$$
Since $q\geqslant 27$, $n_p(S)>q^3\geqslant 3q^2\geqslant(q+1)^2\geqslant p^2$, a contradiction. 

If $p\mid q-1$. Applying \cite[Table 8.43]{Bray} again, $S$ has a maximal subgroup $M$ isomorphic to $[q^3]:C_{q-1}$. Let $Q$ be a Sylow $p$-subgroup of $M$. Then $N_M(Q)=C_{q-1}$ by using \cite[Table 5]{Vasiliev}, thus $n_p(M)=|M:N_M(Q)=|M:C_{q-1}|=q^3$. Since $n_p(S)\geqslant n_p(M)$ and $q^3>(q-1)^2\geqslant p^2$, $n_p(S)\geqslant p^2$, a contradiction. 

We now suppose that $p\mid q-\sqrt{3q}+1$. By using \cite[Table 8.43]{Bray}
again, $S$ has a maximal subgroup $M$ of type $C_{q-\sqrt{3q}+1}:C_6$. In the light of Lemma \ref{cyclicmaximal}, $N_S(P)=M$. Therefore,
\begin{align*}
n_p(S)=|S:M|=\frac{q^3(q^2-1)(q+\sqrt{3q}+1)}{6}\geqslant q^4.
\end{align*}
Note that $p^2\leqslant (q-\sqrt{3q}+1)^2\leqslant 4q^2< q^4$, so $n_p(S)\geqslant p^2$. Finally,
we assume that $p\mid q+\sqrt{3q}+1$. Similarly as above, we have
\begin{align*}
  n_p(S)=\frac{q^3(q^2-1)(q-\sqrt{3q}+1)}{6}\geqslant q^4.
\end{align*}
Observe that $p^2\leqslant (q+\sqrt{3q}+1)^2\leqslant 9q^2$ since $q\geqslant 27$.
Thus $n_p(S)\geqslant q^4\geqslant 9q^2\geqslant p^2$. These cases are impossible.

\end{proof}

\begin{prop}\label{{}^3D_4(q)simple}
If $S\cong {}^3D_4(q)$, then there does not exist $p\in \pi(S)$ such that $n_p(S)<p^2$.
\end{prop}

\begin{proof}
Suppose $S\cong {}^3D_4(q)$, then $$|S|=q^{12}(q^2+q+1)^2(q^2-q+1)^2(q^4-q^2+1)(q+1)^2(q-1)^2.$$ Since the Sylow $p$-subgroup $P$ of $S$ is isomorphic to $C_p$, $p$ divides $q^4-q^2+1$. By \cite [Table 8.51]{Bray}, $S$ has a maximal subgroup $M$ of type $C_{q^4-q^2+1}.C_4$. In the light
of Lemma \ref{cyclicmaximal} we have
\begin{align*}
 n_p(S)=|S:M|=\frac{q^{12}(q^2+q+1)^2(q^2-q+1)^2(q+1)^2(q-1)^2}{4}\geqslant 8q^8.
\end{align*} 
Clearly, $p^2\leqslant (q^4-q^2+1)^2\leqslant 8q^8$, so $n_p(S)\geqslant p^2$, a contradiction.
\end{proof}

\begin{prop}\label{F_4(q)simple}
If $S\cong F_4(q)$, then there does not exist $p\in \pi(S)$ such that $n_p(S)<p^2$.
\end{prop}

\begin{proof}
Assume $S\cong F_4(q)$, then $$|S|=q^{24}(q^6-1)^2(q^2-1)^2(q^2+1)^2(q^4+1)(q^4-q^2+1).$$ Since the Sylow $p$-subgroup $P$ of $S$ is isomorphic to $C_p$, either $p\mid q^4+1$ or $p\mid q^4-q^2+1$ by Lemma \ref{sylowstructure}, so that $p^2\leqslant (q^4+1)^2\leqslant 3q^8$. Since $N_S(P)<S$, there exists a maximal subgroup $M$ of $S$ such that $N_S(P)\leqslant M$. Therefore, $n_p(S)=|S:N_S(P)|\geqslant |S:M|$, which is not less than the degree of minimal permutation representation of $F_4(q)$. In the light of \cite [Theorem 2]{vasily}, 
$$n_p(S)\geqslant \frac{(q^{12}-1)(q^4+1)}{(q-1)}\geqslant q^{15}>3q^8.$$ 
Thus $n_p(S)\geqslant p^2$, a contradiction.
\end{proof}

\begin{prop}\label{{}^2F_4(q)simple}
If $S\cong {}^2F_4(q)$ where $q=2^{2m+1}\geqslant 8$, then there does not exist $p\in \pi(S)$ such that $n_p(S)<p^2$.
\end{prop}

\begin{proof}
Let $S\cong {}^2F_4(q)$. Then 
$$|S|=q^{12}(q^4-1)^2(q^2-q+1)(q^2+q+1+\sqrt{2q}(q+1))(q^2+q+1-\sqrt{2q}(q+1)).$$ 
Since $P\cong C_p$, $p$ is a divisor of one of the following numbers: $q^2+q+1+\sqrt{2q}(q+1)$, $q^2+q+1-\sqrt{2q}(q+1)$ or $q^2-q-1$. This implies that $p^2\leqslant (q^2+q+1+\sqrt{2q}(q+1))^2\leqslant 25q^4$.
On the other hand, since $S$ is a simple group, $N_S(P)$ must be contained in one of maximal subgroups of $S$. Therefore, $|S:N_S(P)|$ is not less than the degree of minimal permutation representation of ${}^2F_4(q)$.
By using \cite[Theorem 5]{vasilyev}, 
$$n_p(S)=|S:N_S(P)|\geqslant
(q^6+1)(q^3+1)(q+1)>q^{10}>25q^4\geqslant p^2,$$ a contradiction.
\end{proof}



\begin{prop}\label{E_6(q)simple}
If $S\cong E_6(q)$, then there does not exist $p\in \pi(S)$ such that $n_p(S)<p^2$.
\end{prop}

\begin{proof}
Let $S\cong E_6(q)$. Then 
$$|S|=\frac{q^{36}(q^{12}-1)(q^9-1)(q^8-1)(q^6-1)(q^5-1)(q^2-1)}{(3,q-1)}.$$
Obviously, $p$ does not divide $q-1$ since the Sylow $p$-subgroup $P$ of $S$ is isomorphic to $C_p$, that is $p\nmid(3,q-1)$. We first assume that $p|(q^9-1)$. By using  
\cite[Table 9]{craven} $S$ has a subgroup $H$ of type $\PSL_3(q^3)$, thus $n_p(S)\geqslant n_p(H)\geqslant p^2$ by Proposition \ref{linearsimple}, it is impossible. Suppose that $p$ is a prime divisor of $(q^{12}-1)(q^8-1)(q^6-1)(q^5-1)(q^2-1)$, then $p^2\leqslant (q^6+1)^2<3q^{12}$. On the other hand, $n_p(S)$ is not less than the degree of the minimal permutation representation of $E_6(q)$ since $N_S(P)<S$. With the help of \cite[Theorem 1] {Vasil} we have 
$$n_p(S)\geqslant \frac{(q^9-1)(q^8+q^4+1)}{q-1}>q^{16}>3q^{12}>p^2,$$
which is a contradiction.
\end{proof}

\begin{prop}\label{{}^2E_6(q)simple}
If $S\cong {}^2E_6(q)$, then there does not exist $p\in \pi(S)$ such that $n_p(S)<p^2$.
\end{prop}

\begin{proof}
Let $S\cong {}^2E_6(q)$. Then
$$|S|=\frac{q^{36}(q^{12}-1)(q^9+1)(q^8-1)(q^6-1)(q^5+1)(q^2-1)}{(3,q+1)}.$$ 
Observe that $p$ does not divide $(3,q+1)$. We first suppose that $p$ divides $q^9-1$. By \cite[Table 10]{craven} $S$ has a subgroup $H$ of type $\PSU_3(q^3)$, so that $n_p(S)\geqslant n_p(H)\geqslant p^2$ with the help of Proposition \ref{unitarysimple}, a contradiction. We next suppose that $p$ divides $(q^{12}-1)(q^8-1)(q^6-1)(q^5+1)(q^2-1)$, then $p^2\leqslant (q^6+1)^2\leqslant 3q^{12}$. Similarly as above, $n_p(S)$ is not less than the smallest index of a maximal subgroup of $S$. In the light of \cite[Theorem 4]{vasilyev}
$$n_p(S)\geqslant \frac{(q^{12}-1)(q^6-q^3+1)(q^4+1)}{q-1}>q^{15}>3q^{12}\geqslant p^2,$$
a contradiction.
\end{proof}

\begin{prop}\label{E_7(q)simple}
If $S\cong E_7(q)$, then there does not exist $p\in \pi(S)$ such that $n_p(S)<p^2$.
\end{prop}

\begin{proof}
Let $S\cong E_7(q)$. Then 
$$|S|=q^{63}(q^{18}-1)(q^{14}-1)(q^{12}-1)(q^{10}-1)(q^8-1)(q^6-1)(q^2-1).$$
Since $p$ is a prime divisor of $|S|$, $p^2\leqslant (q^9+1)^2<3q^{18}$. $S$ is a simple group, so $N_S(P)$ must be contained in one maximal subgroups of $S$. Therefore, $n_p(S)=|S:N_P(S)|$ is not less than the smallest index of a maximal subgroup of $S$. By \cite[Theorem 2]{Vasil}, 
$$n_p(S)\geqslant \frac{(q^{14}-1)(q^9+1)(q^5+1)}{(q-1)}\geqslant q^{27}>3q^{18}>p^2,$$
which is a contradiction.
\end{proof}

\begin{prop}\label{E_8(q)simple}
If $S\cong E_8(q)$, then there does not exist $p\in \pi(S)$ such that $n_p(S)<p^2$.
\end{prop}

\begin{proof}
If $S\cong E_8(q)$, then
$$|S|=q^{120}(q^{30}-1)(q^{24}-1)(q^{20}-1)(q^{18}-1)(q^{14}-1)(q^{12}-1)(q^8-1)(q^2-1).$$ 
Clearly, $p^2\leqslant (q^{15}+1)^2<3q^{30}$ and $n_p(S)$ is not less than the smallest index of a maximal subgroup of $S$. By using \cite[Theorem 3]{Vasil}
$$n_p(S)\geqslant \frac{(q^{20}-1)(q^{12}+1)(q^{10}+1)(q^6+1)}{(q-1)}> q^{47}>3q^{30}>p^2,$$
a contradiction.
\end{proof}

\section{Proof of Theorem \ref{main}}

In this section, we will prove Theorem \ref{main} by the following theorem due to M. Hall.

\begin{thm}[Theorem 2.2 of \cite{MHJ}]\label{MHJ}
The number $n_p(G)$ of Sylow $p$-subgroups in a finite group $G$ is the product of factors of the following two kinds:
\begin{enumerate}
    \item the number $s_p$ of Sylow $p-$subgroups in a simple group $X$;
    \item a prime power $q^t$ where $q^t\equiv 1(\bmod~p)$.
\end{enumerate}
\end{thm}

Let $G$ be a finite group with $n_p(G)<p^2$. Assume that $n_p(G)$ is not a power of prime. In order to prove Theorem \ref{main}, it is enough to prove that  $n_p(G)=n_p(S)$ for some simple group $S$ with $n_p(S)<p^2$. By Theorem \ref{MHJ} we have $$n_p(G)=q_1^{t_1}q_2^{t_2}\cdots q_m^{t_m}\cdot n_p(S_1)n_p(S_2)\cdots n_p(S_n)$$ where $S_i$ is a simple group for $1\leqslant i\leqslant n$ and $q_j^{t_j}\equiv 1(\bmod~p)$ for $1\leqslant j\leqslant m$. Obviously, if $q_j^{t_j}$ and $n_p(S_i)$ are not equal to $1$ then they are greater than $p$, for any $i,j$. Since $n_p(G)<p^2$ is not a power of prime, $q_1^{t_1}\cdots q_m^{t_m}=1$ and there exists at most one index $i$ such that $n_p(S_i)\neq 1$. Therefore $n_p(G)=n_p(S_i)$ where $S_i$ is a simple group with $n_p(S_i)<p^2$.

Conversely, a power of a prime $r$ which is congruent to 1 mod $p$ can become a Sylow $p$-number. In fact it is one of a Frobenius group with the complement of the cyclic group of order $p$ and the kernel of elementary abelian $r$-group. \hfill$\square$

\section{The proof of Theorem \ref{maincorollary}}

In this section, we will prove Theorem \ref{maincorollary}. Recall that, if $G$ is a $p$-solvable group, meaning that it admits a series of normal subgroups $1=V_0<\ldots<V_n=G$ such that each factor $V_{i+1}/V_i$ is either a $p$-group or a $p'$-group. It is easy to see that normal subgroups and the quotient groups of $p$-solvable group are also $p$-solvable.

\begin{lemma}\label{equ}
    Let $a=1+r_1p$, $b=1+r_2p<p^2$ where $r_1, r_2$ are integers. If $a|b$, then $a=1$ or $a=b$.
\end{lemma}

\begin{proof}
    We assume that $a\neq 1$, that is $r_1\neq 0$. Since $a$ divides $b$, we can write $1+r_2p=d(1+r_1p)$ where $d$ is an integer. Thus $r_2=\frac{d-1}{p}+dr_1$. Note that $r_2$, $dr_1$ are integers, so $\frac{d-1}{p}$ is also an integer, this implies that $d=1$ or $d>p$. If $d>p$, then $1+r_2p=d(1+r_1p)>p^2$, a contradiction. Therefore $d=1$. 
\end{proof}

First, we prove that if $n_p(G)$ is a power of a prime then $G$ is $p$-solvable. To prove our result, we can prove that if $G$ is not a $p$-solvable group, then $n_p(G)$ is not a power of a prime. Since $G$ is not a $p$-solvable group, $G$ is nonsolvable. Therefore, 
$G$ has a nonsolvable chief factor $S^m$ such that $p$ divides $|S|^m$ where $S$ is a nonabelian simple group and $m\geqslant 1$. Obviously, $n_p(S)\neq 1$. In the light of Lemma \ref{normal} we have $n_p(S)$ divides $n_p(G)$. Moreover, we can get $n_p(S)=n_p(G)$ by Lemma \ref{equ}. On the other hand, $n_p(G)<p^2$, thus $n_p(S)$ satisfies one of items (1), (2) and (3) of Theorem \ref{sylowsimple}. If $n_p(S)=1+p=q^a$ for some prime $q$ and integer $a$, then $q=2$, $a$ is a prime number, that is $p=2^a-1$ is a Mersenne prime number, a contradiction. If $n_p(G)=1+\frac{(p-3)p}{2}=q^a$ where $p$ is a Fermat prime number, $q$ is prime and $a$ is an integer, then $p^2-3p+2=2q^a$. By Lemma \ref{number}, we have $p=3$, contradicting $p>3$. Therefore, $n_p(G)$ is not a power of a prime. This proves that if $G$ is not a $p$-solvable group then $n_p(G)$ is not a power of a prime.

Conversely, let $G$ be a $p$-solvable group. We assume that $G$ is a minimal order such that $n_p(G)$ is not a power of a prime. Write $n_p(G)=q^{t_1}_1q^{t_2}_2\cdots q^{t_k}_k$ is the prime factorization of $n_p(G)$. We claim that $G$ is nonsolvable. If this were not the case, then  $q^{t_i}_i\equiv 1(\bmod ~p)$ for any $i=1,\ldots,k$ (see \cite{PHall}). 
By our assumption that $n_p(G)$ is not a power of a prime, $n_p(G)$ has at least two prime divisors. This implies that $n_p(G)>p^2$, a contradiction.
It is not difficult to see that $G$ is not simple because otherwise $p\not\in \pi(G)$, so that $n_p(G)=1$, a contradiction.

Let $M$ be a maximal normal subgroup of $G$, then either $G/M$ is a cyclic group of order $p$ or $G/M\cong S$ where $p$ does not divide $|S|$. Thus $n_p(G/M)=1$. By Lemma \ref{normal} we have 
$$n_p(G)=n_p(G/M)n_p(M)n_p(N_{PM}(P\cap M)/P\cap M),$$ 
where $P$ is a Sylow $p$-subgroup of $G$. Observe that $n_p(G)$ is divisible by $n_p(M)$, so $n_p(M)=1$ or $n_p(M)=n_p(G)$ with the help of Lemma \ref{equ}. Since $M$ is a $p$-solvable group and $|M|<|G|$, by the minimality of $G$, $n_p(M)=1$. Thus 
$$n_p(G)=n_p(N_{PM}(P\cap M)/P\cap M).$$

Assume $p$ does divide $|G/M|$, then $P\leqslant M$. Moreover, $P\unlhd M$. Thus $$N_{PM}(P\cap M)/P\cap M=N_M(P)/P=M/P,$$ this implies that $p\not\in \pi(M/P)$, thus $n_p(G)=n_p(M/P)=1$, contradicting the fact that $n_p(G)$ is not a power of a prime.

Assume $G/M\cong C_p$, then $P\cap M$ is a Sylow $p$-subgroup of $M$. Since $n_p(M)=1$, $P\cap M\unlhd M$. Obviously, $P\cap M\unlhd P$ because $M$ is normal in $G$, thus $N_{PM}(P\cap M)=PM$. Therefore $$N_{PM}(P\cap M)/P\cap M=PM/P\cap M.$$ Now we have a normal series $$P\cap M\unlhd M\unlhd PM,$$ thus $M/P\cap M\unlhd PM/P\cap M$. Since $p$ does not divides $|M/M\cap P|$, the Sylow $p$-subgroup of $PM/M\cap P$ is $P/M\cap P$, which is isomorphic to $C_p$ because $P/M\cap P\cong MP/M=G/M$. Note that 
$$PM/P\cap M\cong (M/M\cap P)(P/M\cap P)\cong (M/M\cap P)C_p$$ 
and $p$ does not divide $M/(M\cap P)$, thus $PM/M\cap P$ is a $p$-solvable group. Recall that $n_p(G)=n_p(PM/M\cap P)$, the minimality of $G$ implies that $|PM/M\cap P|=|G|$. Therefore $G=MP$ and $P\cap M=1$, that is $$G\cong M\rtimes P\cong M\rtimes C_p.$$ 
Observe that $(|M|,|P|)=1$, so the action of $P$ on $M$ is coprime. For any $r\in \pi(M)$, there exists a $P$-invariant Sylow $r$-subgroup $R$ of $M$ by \cite [8.2.3(a)] {HKBS}, this implies that $RP\leqslant G$.
In the light of \cite [Theorem A] {GN}, $n_p(RP)$ divides $n_p(G)$, thus $n_p(RP)=1$ or $n_p(RP)=n_p(G)$ with the help of Lemma \ref{equ}. Assume $n_p(G)=n_p(RP)$. Since $RP$ is a $p$-solvable group, by the minimality of $G$ we can get $RP=G$, contradicting the nonsolvability of $G$. Hence $n_p(RP)=1$, implying that $RP=R\times P$. Since $$M=\prod\limits_{r\in \pi(M)}R,$$ $G=M\times P$, and hence $n_p(G)=1$, a contradiction. \hfill$\square$


\end{document}